\documentclass{article}

\usepackage[utf8]{inputenc}
\usepackage{amsmath,amsfonts,amssymb,amsthm,pb-diagram,picinpar,graphicx,color}
\usepackage{algorithm}
\usepackage{amscd}
\usepackage{algorithmic}
\usepackage{colortbl}
\usepackage{enumerate}
\usepackage{hyperref}
\usepackage{bm}
\usepackage{tikz}
\usepackage{pgf}
\usepackage[T1]{fontenc}
\usepackage{enumitem}
\usepackage{comment}
\usepackage{svg}
\usetikzlibrary{shapes,matrix,calc,arrows,decorations.markings,knots,patterns,intersections,cd}

\usepackage{subcaption}
\newcommand{\CC}{\mathbb{C}}
\newcommand{\ZZ}{\mathbb{Z}}
\newcommand{\PP}{\mathbb{P}}

\newcommand{\mcC}{\mathcal{C}}
\newcommand{\mcB}{\mathcal{B}}

\newcommand{\mcL}{\mathcal{L}}

\newcommand{\mcD}{\mathcal{D}}
\newcommand{\mcO}{\mathcal{O}}
\newcommand{\mcR}{\mathcal{R}}
\newcommand{\mcQ}{\mathcal{Q}}

\newcommand{\Kbar}{\overline{K}}

\newcommand{\fd}{\mathfrak{d}}

\newcommand{\bb}{\boldsymbol{b}}

\newcommand{\vc}{\boldsymbol{c}}
\newcommand{\ve}{\boldsymbol{e}}
\newcommand{\vzero}{\boldsymbol{0}}

\newcommand{\divi}{\mathop{\mathrm{div}}\nolimits}

\newcommand{\Supp}{\mathop{\mathrm{Supp}}\nolimits}

\newcommand{\sr}{\mathop{\mathrm{sr}}\nolimits}
\newcommand{\red}{\mathop{\mathrm{r}}\nolimits}

\newcommand{\Red}{\mathop{\mathrm{Red}}\nolimits}

\newcommand{\ord}{\mathop{\mathrm{ord}}\nolimits}

\newcommand{\tot}{\mathrm{t}}

\newcommand{\Cmb}{\mathop{\mathrm{Cmb}}\nolimits}

\newcommand{\I}{\mathop {\mathrm {I}}\nolimits}

\newtheorem{thm}{Theorem}[section]
\newtheorem{lem}[thm]{Lemma}     

\newtheorem{prop}[thm]{Proposition}

\theoremstyle{definition}
\newtheorem{defin}[thm]{Definition}

\newtheorem{ex}[thm]{Example}

\theoremstyle{remark}
\newtheorem{rem}[thm]{Remark}

\renewcommand{\thesubparagraph}{\theparagraph.\@arabic\c@subparagraph}
\usepackage{subcaption}
\begin{document}

\title{\bfseries 
A curve of genus $2$ over $\CC(t)$ and hexads of  bitangents to a smooth quartic}

\author{\bfseries Shinzo BANNAI\footnote{Partially supported by JSPS KAKENHI Grant Number JP23K03042} \, and Hiro-o TOKUNAGA\footnote{Partially supported by JSPS KAKENHI Grant Numbers JP23H00081 and JP24K06673}}

\date{\empty}
\maketitle

\begin{abstract}
In this paper we study the Abel-Jacobi map of a genus 2 curve over $\CC(t)$ arising from a double cover of $\PP^2(\CC)$ branched along six lines. 
We utilize the so-called Mumford representations to conduct the calculations.
As an application, we give an alternative proof of classical results on quartics and bitangents.
Furthermore, we obtain a method to give explicit defining equations of quartic curves that have the six lines as bitangents and apply it to determine the number of connected components of the realization space.
\end{abstract}

\section*{Introduction}

All curves and surfaces in this article are defined over the field of complex numbers $\CC$ unless
otherwise stated.
Let $\mcB$ be a reduced plane curve of even degree and let $D$ be another reduced plane curve having 
no common components with $\mcB$. In our previous papers \cite{bst24, bty25} and the references therein,
we have studied the embedded topology of $\mcB + D$ (or the topology of 
the pair $(\PP^2, \mcB+D)$) via the double cover of $\PP^2$ branched along $\mcB$ and the pull-back of $D$ by the
double cover.
We called such an approach {\it \lq arithmetic of double covers\rq} (cf.\cite{bst24}). Let us explain it more precisely.

Let $f'_{\mcB} : S'_{\mcB} \to  \PP^2$ be a double cover
branched along $\mcB$ and let $\mu : S_{\mcL} \to S'_{\mcL}$ be the canonical resolution (see \cite{horikawa} for the canonical resolution). Choose a point $z_o \in \PP^2\setminus\{\text{singular points of $\mcB$}\}$.  The pencil 
$\overline{\Lambda}_{z_o}$ of lines through
$z_o$ on $\PP^2$ gives rise to a pencil $\Lambda_{z_o}$ of  hyperelliptic curves of genus $\frac 12\deg B - 1$ on $S_{\mcB}$.
Let $\nu_{z_o}: S_{\mcB, z_o} \to S_{\mcB}$ be the resolution of indeterminacy of $\Lambda_{z_o}$. We denote the 
 induced hyperelliptic fibration of genus $\frac 12 \deg B - 1$  by
$\varphi_{\mcB, z_o}: S_{\mcB, z_o} \to \PP^1$. The diagram below explains how we obtain the surfaces $S'_{\mcB}, S_{\mcB}$ and $S_{\mcB, z_o}$: 



\[
\begin{CD}
S'_{\mcB} @<{\mu_{\mcB}}<< S_{\mcB} @<{\nu_{z_o}}<<S_{\mcB, z_o} \\
@V{\pi'_{\mcB}}VV                 @VV{\pi_{\mcB}}V         @VV{\pi_{\mcB, z_o}}V \\
\PP^2@<<{q}< \widehat{\PP^2} @<<{q_{z_o}}< (\widehat{\PP^2})_{z_o},
\end{CD}
\]
where  $q$ is a composition of blowing-ups so that the branch locus of
$\pi_{\mcB}$ becomes smooth and $q_{z_o}$ is a composition of two blowing-ups 
(resp. a blowing-up) if $z_o \in \mcB$ (resp. $z_o \not\in \mcB$) . 
We put $\tilde{\pi}_{\mcB} = \pi'_{\mcB}\circ\mu_{\mcB} = q\circ \pi_{\mcB}$ and
$\tilde{\pi}_{\mcB,z_o} = \pi'_{\mcB}\circ\mu_{\mcB}\circ\nu_{z_o} = q\circ q_{z_o}\circ \pi_{\mcB, z_o}$.
We also denote the covering transformations of the above three double covers by $\sigma_{\pi'_{\mcB}}, \sigma_{\pi_{\mcB}}$ and $\sigma_{\pi_{\mcB, z_o}}$, respectively.
If $z_o \in \mcB$ (resp. $z_o \not\in \mcB$), The exceptional curve of the first blowing-up of $q_{z_o}$ gives rise to a section $O$ 
(resp. sections $O^{\pm}$). 
In this article, we only consider the case of $z_o \in \mcB$. We denote the generic fiber of $\varphi_{\mcB, z_o}$
by $\mcC_{\mcB, z_o}$. The section $O$ gives rise to  a rational point on $\mcC_{\mcB, z_o}$. By abuse of notation,
we also denote this rational point by
$O$.
In order to study the topology of $(\PP^2, \mcB + D)$, we take the pull-back
$\tilde{\pi}_{\mcB}^*D$ or $\tilde{\pi}_{\mcB, z_o}^*D$ and study its geometric
and arithmetic properties. This means for example  the combinatorics of the pull-back and the arithmetic of the induced
divisor on the generic fiber $\mcC_{\mcB, z_o}$. This is what we call \lq {\it arithmetic 
of double covers}'.

In our previous articles such as
e.g., \cite{absst, bty20, bannai-tokunaga20-2, bannai_ohno, bty25},
$S_{\mcB, z_o}$ is a rational elliptic
surface and we have intensively used properties of the group of sections
based on the theory of Mordell-Weil
lattices and so on. In contrast to these cases, in this article, we consider the case where the generic fiber
is a curve of genus $2$. Yet this article can be 
considered as a continuation of our previous articles such as \cite{bty20, bannai_ohno,bty25}. 
In fact, we focus on a smooth quartic and hexads of its bitangents such that twelve tangent points
are on a cubic. Let us explain our object precisely.

Let $\mcQ$ be a smooth plane quartic in $\PP^2$. A bitangent to $\mcQ$ is a 
line tangent to $\mcQ$ at two distinct points or at one point with multiplicity $4$. We call the former ordinary and the latter special (or inflectional bitangent). 
It is well-known that there exist 27 bitangents to $\mcQ$ and they have been
studied by a lot of mathematicians (see \cite[Chapter 6]{dolgachev} and the
references therein). 
In this article, we consider hexads of ordinary bitangents to $\mcQ$ such
that the $12$ tangent points lie on a reduced cubic. Let $\mcQ$ be a smooth quartic
and let $\mcL = \sum_{i = 0}^5L_i$ be the union of six bitangents of $\mcQ$.
$C_3$ denotes the cubic through $\mcQ\cap \mcL$. We assume that
\begin{center}
    ($\ast$) None of the three lines in $\mcL$ are concurrent.
\end{center}

The \lq arithmetic of double covers' framework can be applied to the above setting by taking $\mcL$ (resp. $\mcQ$ ) as $\mcB$ (resp. $D$).
Under these circumstances,
$\mcQ$ gives rise to two horizontal irreducible curves $\mcQ^{\pm}$ meeting a fiber of $\varphi_{\mcL, z_o}$ at $4$ points. 
This means that $\mcQ$ is a splitting curve
with respect to $\tilde{\pi}_{\mcL}$ which is defined
in \cite[Definition 0.1]{tokunaga10}. 
\begin{rem}\label{rem:splitting}
$\mcQ$ is a splitting curve with respect to $\mcL$ if and
only if there exists a cubic through the twelve bitangent points. This follows from the fact that $\mcQ$ is splitting
with respect to $\tilde{\pi}_{\mcL}$ if and only if 
twelve intersection points give rise to a divisor
linearly equivalent to the one induced by $3L_0$. Also, this fact can be confirmed using a criterion for calculating the splitting number given by T.~Shirane in \cite[Theorem2.7]{shirane17}.
\end{rem}

Now let $\mcC_{\mcL, z_o}$ be the 
generic fiber of $\varphi_{\mcL, z_o}$ as above. 
$\mcC_{\mcL, z_o}$ is a curve of genus $2$ over $\CC(t)$. Note that this is a new
feature of this article as we worked with elliptic curves
over $\CC(t)$ in previous articles.
Let $\fd_{\mcQ^{\pm}} = \mcQ^{\pm}|_{\mcC_{\mcL,z_o}}$ be the divisor of degree $4$ on $\mcC_{\mcL, z_o}$. Our main tool to deal with $\fd_{\mcQ^{\pm}}$ is so-called the Mumford representation
of  divisors on a hyperelliptic curve, which was given in \cite{mumfordII} and has been exploited by those working on 
hyperelliptic curve cryptography (see \cite{galbraith, MWZ}). 
Mumford representations allow us to calculate the Abel-Jacobi map of the generic fiber  $\mcC_{\mcL, z_o}$.
Our main purpose is to study reduced plane curves with
the combinatorics $\mcL + \mcQ$  through these calculations enabled by these tools. We start with a set of six lines and construct a quartic $\mcQ$ that have the six lines as bitangents, through the calculations of the Abel-Jacobi map.


The combinatorics of a reduced (possibly reducible) plane curve $C$ is roughly given by the data of the number, degrees, singularities of the irreducible components and how they intersect with each other (see \cite{survey}, \cite{abst2023-1} for the precise definition). Many mathematicians have studied the interplay between the combinatorics of $C$ and the embedded topology of $C$, i.e. the topology of the pair $(\PP^2, C)$ of topological spaces. It is known that the combinatorics determines the embedded topology of $C$ in its regular neighborhood (see \cite{shirane26} for a rigorous proof), but does not determine the global embedded topology. This fact leads to the notion of a {\it Zariski pair} defined by E.~Artal in \cite{artal94}. 
To fully understand Zariski pairs, it is important to study the whole set of curves with a given fixed combinatorial type, which we denote by  $\Cmb_C$. For curves of degree $d$, this set can be realized as the set of defining polynomials of $C$ with combinatorial type $\Cmb_C$ up to a non-zero constant multiple viewed as a subset of $\PP^{\frac{d(d+3)}{2}}$. We call this set the {\it realization space} of curves with combinatorial type $\Cmb_C$ and denote it by $\mcR(\Cmb_C)$. The realization space for line arrangements, conic-line arrangements have been studied \cite{absst, bty25}, and for some cases the number of connected components of the spaces have been determined.


Let $\Cmb_{\mcQ\mcL}$ be the combinatorial type of quartics and hexads of ordinary bitangents with the assumption ($\ast$). We consider the subset of $\mcR(\Cmb_{\mcQ\mcL})$ consisting of curves $\mcQ+\mcL$ such that  a cubic through the 12 points of tangency exists and denote it by
\[
\mcR(\Cmb_{\mcQ\mcL})^{C_3}\subset \mcR(\Cmb_{\mcQ\mcL}) \subset \PP^{\frac{10(10+3)}{2}}=\PP^{65}.
\]
Our first result concerning  $\mcQ+\mcL\in \mcR(\Cmb_{\mcQ\mcL})^{C_3}$ is as follows:

\begin{thm}\label{thm:main1} Choose $\mcQ + \mcL \in \mcR(\Cmb_{\mcQ\mcL})^{C_3}$. Then
either (I) there exists a conic which is tangent to all six bitangents or (II) 
the six bitangents are
    distributed into three pairs such that the intersection points of the two lines in each pair are collinear.
The two cases do not occur simultaneously.
\end{thm}

According to G. Salmon
\cite[Article 262]{salmon}, Theorem~\ref{thm:main1} was already known. 
Moreover, for a fixed smooth quartic $\mcQ$, there exist $6048$ hexads of such bitangents $\mcL$ such that
$\mcQ + \mcL \in \mcR(\Cmb_{\mcQ\mcL})^{C_3}$; 1008 hexads of bitangents for (i) and 5040 hexad of bitangents
for (ii).  Theorem~\ref{thm:main1} gives an alternative proof for a part of the statement in \cite{salmon} from
a new point of view.
As there are numerous hexads, even for a fixed quartic $\mcQ$, the above fact in \cite{salmon} depicts the complexity of $\mcR\left(\Cmb_{\mcQ\mcL}\right)^{C_3}$. We will give an explicit description of $\mcR\left(\Cmb_{\mcQ\mcL}\right)^{C_3}$ by using a curve of genus $2$ over $\CC(t)$, which arises from
a double cover of $\PP^2$ ramified along $\mcL$. 
In fact, the Mumford representation of $\fd_{\mcQ^{\pm}}$ 
makes it possible  to describe $\mcR(\Cmb_{\mcQ\mcL})^{C_3}$ with
concrete parameters.
Now we state our second result:
\begin{thm}\label{thm:main2}
The  subset  $\mcR\left(\Cmb_{\mcQ\mcL}\right)^{C_3}$ has exactly two connected components 
$\mcD_i\, (i = 1, 2)$. The two components correspond to the two cases described above. Furthermore, if $\mcB_i \in D_i$ ($i =1, 2$),
$(\mcB_1, \mcB_2)$ is a Zariski pair, i.e, there exists no homeomorphism between $(\PP^2, \mcB_1)$
and $(\PP^2, \mcB_2)$.     
\end{thm}

\begin{rem}
Zariski pairs of curves consisting of a quartic $\mcQ$ and hexad $\mcL$ where the existence of the cubic is not assumed, have been considered in \cite{bannai_ohno}, \cite{shimada22}. In particular, the number of 
components of the realization space for a smooth quartic $\mcQ$ and hexads of bitangents $\mcL$ is determined in \cite{shimada22}.  In fact, it is known that $\mcR\left(\Cmb_{\mcQ\mcL}\right)$ has 10 connected components, and the subset $\mcR\left(\Cmb_{\mcQ\mcL}\right)^{C_3}$ that we considered is exactly the union of two of the connected components. Yet
our approach to prove Theorem~\ref{thm:main2} through an explicit description of the two connected components with
concrete defining equations is new and we hope that this viewpoint is still worthwhile for this article to be published.
\end{rem}

This article consists of $4$ sections. In Section~\ref{sec:preliminaries}, we explain how we deal with divisors on a hyperelliptic curve via Mumford representations. We prove Theorem~\ref{thm:main1} based
on our approach 
in Section~\ref{sec:genus_2_curve}. We give explicit examples with parameters in Section~\ref{sec:construction} and prove Theorem~\ref{thm:main2} in Section~\ref{sec:proof_theorem}.

\section{Some preliminaries on a hyperelliptic curve of genus \texorpdfstring{$2$}{}}
\label{sec:preliminaries}
%
%
We refer to \cite{galbraith, MWZ, takatoku20-2} and \cite[Appendix]{bst24} as main references.

Let $K$ be a field of characteristic zero and $\Kbar$ be its algebraic closure.
Let $\mcC$ be a hyperelliptic curve of genus $2$ defined over $K$ given by
an equation of the form
\[
\mcC: y^2 = f_{\mcC}(x), \quad f_{\mcC}(x) = x^5 + c_1 x^4 + \ldots+ c_5, \quad c_i \in K,
\]
where $f_{\mcC}(x) = 0$ has no multiple solutions in $\Kbar$. We denote the point at infinity by $O$ and 
put $\mcC^a:= \mcC \setminus\{O\}$.
Let $\Kbar[\mcC] := \Kbar[x, y]/\langle y^2 - f_{\mcC}\rangle$ (resp.
 $K[\mcC]:= K[x, y]/\langle y^2 - f_{\mcC}\rangle)$ be the coordinate ring over $\Kbar$ (resp.
 coordinate ring over $K$) and we write
 $\phi_{\mcC} : \Kbar[x, y] \to \Kbar[\mcC]$ (resp. $\phi_{\mcC, K} :
 K[x, y] \to K[\mcC]$). For $g \in \Kbar[x, y]$, $[g]$ means its  class in 
 $\Kbar[\mcC]$ (The same notation is used for $g \in K[x, y]$). The rational function field over $\Kbar$ (res. over $K$) 
 is denoted by $\Kbar(\mcC)$ (resp. $K(\mcC)$). For $P \in \mcC$,
 we denote the local ring at $P$ by $\mcO_P$. By choosing 
 a local uniformizing parameter, we define 
 the order function at $P$ by $\ord_P$.
 We denote the hyperelliptic involution by $\iota: \mcC \to \mcC, (x, y) \mapsto (x, -y)$.
We write the quotient morphism by 
$\pi : \mcC \to  \mcC/\langle \iota \rangle \cong \PP^1$. For a divisor 
$\fd = \sum_{P\in \mcC}n_P P$, we put $\Supp(\fd) = \{P \in \mcC \mid n_P\neq 0\}$.
\subsection{Semi-reduced divisors and their  reduction}

We here introduce the notion of a semi-reduced divisor, which plays an important role to describe
the addition on the Jacobian of $\mcC$. It is our main tool  to produce 
curves with prescribed properties.

We call a divisor $\fd$ {\it affine} if $O \not\in \Supp(\fd)$.
For an affine effective divisor $\fd$ on $\mcC$,  we have a unique decomposition
$\fd = \fd_{\sr} + \fd_o$ as follows:
\begin{enumerate}
    \item[(i)] there exists
    an effective divisor $\bar{\fd}_o$ on $\PP^1\setminus \{\pi(O)\}$ such that $\fd_o = \pi^*\bar{\fd}_o$
    and
    \item[(ii)] $\fd_{\sr}-\pi^\ast\bar{P}$ is not effective for any $\bar{P}\in\PP^1$.
\end{enumerate}

\begin{defin}\label{def:semi-reduced}
Let $\fd$ be an affine effective divisor on $\mcC$ and let $\fd = \fd_{\sr} + \fd_o$ be
the decomposition as above.
\begin{enumerate}
    \item[(i)]  $\fd$ is said to be \textit{semi-reduced} if $\fd = \fd_{\sr}$.
    \item[(ii)] A semi-reduced divisor is said to be \textit{totally-reduced} (\textit{$\tot$-reduced}, for 
    short) if $\deg \fd \le g$.
\end{enumerate}   
\end{defin}
\begin{rem}\label{rem:reduced}{In \cite{cantor87,  galbraith, MWZ}, a $\tot$-reduced divisor is simply called
\lq reduced\rq. We use the terminology \lq totally-reduced\rq\,  to avoid confusion, since 
 a divisor $\sum_P n_P P$ with $n_P = 1$ for $P \in \Supp(\fd)$ is also said to be reduced
(for example see \cite{iitaka}).
}
\end{rem}

Here are some properties about semi-reduced divisors. 
See \cite{galbraith, MWZ} for details.

\begin{prop}\label{prop:semi-reduced} Under the above settings, we have:
\begin{enumerate}
 \item[\rm(a)] 
For any divisor $\fd$, there exists an algorithm to
 find a semi-reduced divisor $\fd_1$ such that
\begin{enumerate}
\item[{\rm (i)}] $\fd - (\deg \fd)O \sim \fd_1 - (\deg\fd_1)O$, and
\item[{\rm (ii)}] $|\fd| \ge |\fd_1| (= \deg \fd_1)$.
\end{enumerate}

\item[{\rm (b)}] 
For any divisor $\fd$, there exists a
unique $\tot$-reduced divisor $\red(\fd)$ such that
\[
\fd - (\deg\fd) O\sim \red(\fd) - (\deg{\red(\fd)})O.
\]

\end{enumerate}
Here we denote the the linear equivalence of  two divisors by $\sim$ and  we also  put $\deg \fd := \sum_P m_P$ and
 $|\fd| := \sum_P|m_P|$. 

\end{prop}

\begin{rem}
    The reduction process from $\fd$ to $\red(\fd)$ given in the above Proposition  amounts to the calculation of the Abel-Jacobi map on $\mcC$.
\end{rem}

\begin{prop}\label{prop:mumford}{
Let $\fd = \sum_{P \in \mcC}n_P P$ be a semi-reduced divisor 
 on $\mcC$. There exists a unique pair
 $(u_{\fd}, v_{\fd})$ of polynomials $u_{\fd}, v_{\fd} \in
 \Kbar[x]$ as follows:
\begin{itemize}
    \item If $\fd = 0$,  then $u_{\fd} = 1, v_{\fd} = 0$. In particular, $u_{\fd}$ 
    is monic.
    \item If $\fd \neq 0$, then $u_{\fd}$ and $v_{\fd}$ satisfy the three conditions below:
    \begin{enumerate}
        \item[\rm{(i)}] $u_{\fd} = \prod_{P \in \Supp(\fd)}(x - x_P)^{n_P}$. 
        \item[\rm{(ii)}] $\deg v_{\fd} < \deg u_{\fd}$ and  $\ord_{P}([y - v_{\fd}]) \ge n_P$ for $\forall P \in \Supp(\fd)$.
  \item[\rm{(iii)}] $v_{\fd}^2 - f_{\mcC}$ is divisible by $u_{\fd}$.

    \end{enumerate}
\end{itemize}
Moreover $\fd$ can be recovered from  $(u_{\fd}, v_{\fd})$ by 
\[
\fd = \sum_{P \in \mcC^a} \min\{\ord_P([u_{\fd}]), \ord_P([y- v_{\fd}])\}P.
\]
}    
\end{prop}

By Proposition~\ref{prop:mumford}, there is a bijection between 
the set of  semi-reduced divisors $\fd$ and that of pairs  $(u_{\fd}, v_{\fd})$ of polynomials. Namely, if a pair $(u, v)$ of polynomials in $\Kbar[x]$ satisfying 
(i), (ii) and (iii) as above are given, then the divisor given by the above form
is a semi-reduced divisor which we  denote it by $\fd(u, v)$. This leads us to
the following definition.

 \begin{defin} \label{def:mumford}{\rm
 Let $\fd$ be a semi-reduced divisor on $\mcC$. The pair $(u_{\fd}, v_{\fd})$ of
 polynomials in Proposition~\ref{prop:mumford} is called the \textit{Mumford representation} of $\fd$.
 }
 \end{defin}

\subsection{Semi-reduced divisors and Gr\"obner bases}

Let $\fd = \sum_P m_P P$ be an affine effective divisor. Define
\begin{eqnarray*}
    I(\fd) & := & \{ \xi \in \Kbar[\mcC] \mid \ord_P(\xi) \ge m_P (\forall P \in \Supp(\fd)) \} \\
    \widetilde{I(\fd)} &:= & \phi_{\mcC}(I(\fd)) = \{g \in \Kbar[x, y] \mid [g] \in I(\fd) \}.
\end{eqnarray*}

Let $\fd$ be an affine effective divisor and let $\fd = \fd_{\sr} + \fd_o$ be 
the decomposition in Definition~\ref{def:semi-reduced}. Then we we have polynomials
$u_o, u_{\sr}$ and $v_{\sr}$ in $\Kbar[x]$ satisfying the conditions below:
\begin{enumerate}
    \item[(i)] $I(\fd_o) = \langle [u_o] \rangle$.
    \item[(ii)] $(u_{\sr}, v_{\sr})$ is the Mumford representation of $\fd_{\sr}$ and $I(\fd_{\sr}) = \langle [u_{\sr}], [y - v_{\sr}]\rangle$
    \item[(iii)] $I(\fd) = \langle [u_ou_{\sr}], [u_o(y - v_{\sr})]\rangle$
\end{enumerate}
As for the reduced Gr\"obner basis of $\widetilde{I(\fd)}$, we have

\begin{itemize}
    \item $\{u_o, y^2 - f_{\mcC}\}$ if $\fd = \fd_o$,
    \item $\{u_{\sr}, y - v_{\sr}\}$ if $\fd = \fd_{\sr}$, and
    \item $\{u_ou_{\sr}, u_o(y - v_{\sr}), y^2 - f_o\}$, $f_o \equiv f_{\mcC} (\bmod u_ou_{\sr})$ otherwise.
\end{itemize}
See \cite{takatoku20-2} or  \cite[Appendix]{bst24} for details. We should note that the conditions (a), (b) and (c) in 
\cite[p.101]{bst24} hold only for a monomial ordering such that
the leading monomial of $y^2 - f_{\mcC}$ is $y^2$.

\begin{rem} By \cite[Lemma 10.3.10]{galbraith} or the uniqueness of  reduced Gr\"ober bases, if $\fd$ is a divisor defined over $K$, we see that $u_o, u_{\sr}, v_{\sr} \in K[x]$. 
\end{rem}

\subsection{Semi-reduced divisors of degree \texorpdfstring{$4$}{}}

Let $\fd$ be a semi-reduced divisor on $\mcC$ and let $(u_{\fd}, v_{\fd})$ be its
Mumford representation. We first note that $\fd - 4O \not\sim 0$. In fact, suppose that there
exists $\xi \in \Kbar[\mcC]$ such that $\divi(\xi) = \fd - 4O$. We may assume $\xi$ is given
of the form $[a_0(x) + a_1(x)y]$. As $-\ord_O\xi = \max\{2\deg a_0, 2\deg a_1 + 5\}$, we see that  $\xi$ is given
by $[a_0(x)]$, where $a_0 \in \Kbar[x], \deg a_0 = 2$. This is impossible as
$\fd$ is semi-reduced.

By its definition,  $\deg u_{\fd} = 4$, $v_{\fd} = 0$ or $\deg v_{\fd} \le  3$. 
Write $v_{\fd}^2 - f_{\mcC} = c u_{\fd}u_1$, where $c\in \Kbar$ and $u_1$ is monic.
Now put $v_1 = - v_{\fd} (\bmod u_1)$. Then $\red(\fd) = \fd(u_1, v_1)$.

Conversely, if a $\tot$-reduced divisor $\fd$ with Mumford representation $(u_{\fd}, v_{\fd})$ is given, then we can find a semi-reduced divisor $\tilde{\fd}$ of degree $4$
in the following way:

\begin{enumerate}
    \item  Define $\tilde{v}: = r(x)u_{\fd} - v_{\fd}$, where $r(x) \in \Kbar[x]$ is chosen
    so that 
  \[
  \deg \tilde{v} = \left \{ \begin{array}{cc}
                             2 & \text{if $\deg \fd = 1$} \\
                             3 & \text{if $\deg \fd = 2$}
                             \end{array} \right .
  \]
  \item put $\tilde{v}^2 - f_{\mcC} = c\tilde{u}u_{\fd}$, where $c \in \Kbar$ and
  $\tilde{u}$ is monic.
\end{enumerate}
Then $\tilde{\fd}:= \fd(\tilde{u}, \tilde{v})$ is a semi-reduced divisor of degree $4$
and $\red(\tilde{\fd}) = \fd$.

\begin{rem}\label{rem:deg4_K_div}
{From the above consideration, we have the following facts: 
\begin{itemize}
    \item If $\fd$ is a  $K$-divisor, then $\red(\fd)$ is also a $K$-divisor.
    \item If  a $t$-reduced divisor $\fd$ is defined over $K$ and $r(x) \in K[x]$, then 
    a semi-reduced divisor of degree $4$ obtained as above is also defined over $K$.
\end{itemize}
}
\end{rem}

\section{A curve of genus \texorpdfstring{$2$}{} over \texorpdfstring{$\CC(t)$}{} and
six lines}
\label{sec:genus_2_curve}

\subsection{Some basic facts}\label{sec:basic_facts}

Let $\varphi: S \to \PP^1$ be a smooth projective surface with
a fibration of genus $2$ curves over $\PP^1$ and let $\mcC_S$ be the generic
fiber of $\varphi$. $\mcC_S$ is a curve of genus $2$ over $K = \CC(t)$, $t$ being an
inhomogeneous coordinate. Here we assume that $\mcC_S$ is 
given by
\[
\mcC_S: y^2 = f_{\mcC_S}(t, x), \, f_{\mcC_S}(t, x) = x^5 + c_1(t)x^4 + \ldots + c_5(t), 
\]
where $c_i(t) \in \CC[t], \, \deg c_i(t) \le i + 1\,  (1\le i \le 5)$. Also
we assume that 
\begin{center}
  the irregularity of $S$ is zero,\, i.e., $S$ is a regular surface.  
\end{center}

\begin{rem} The above condition is satisfied if $S$ is birational to a double cover
of $\PP^2$ branched along a plane curve of even degree with at most simple singularities.
As for the irregularity of a double cover or a cyclic cover in general, see \cite{catanese-ciliberto}.    
\end{rem}

Let $D$ be a horizontal
divisor on $S$ and let  $\fd_D$ be the divisor on $\mcC_S$ given by $D|_{\mcC_S}$. Then 
$\fd_D$ is a $K$-divisor on $\mcC_S$.  Assume that $\deg \fd_D > 2$. There
exists a unique $\tot$-reduced divisor $\red(\fd_D)$ such that 
\[
\fd_D - \deg(\fd_D)O \sim \red(\fd_D) - \deg(\red(\fd_D))O.
\]
Since $\red(\fd_D)$ is also a $K$-divisor, there exists a 
horizontal divisor $\red(D)$ corresponding to $\red(\fd_D)$ (see (\cite[Proposition 3.4, Chap. 8]{qing_liu}) and there exists a rational function $\xi \in K[\mcC_S]$ such that 
  $\divi(\xi) = \fd_D + \iota^*\red(\fd_D) - \deg(\fd_D + \red(\fd_D))O$.
    
As $\xi$ can be considered  as an element in $\CC(S)$ and $S$ is a regular surface, we have
\[
D - (D\cdot F)O \sim \red(D) - (\red(D)\cdot F)O + \text{curves in fibers}.
\]
In order to describe this relation more precisely, we use the results
 in \cite{shioda-higher-genus1, shioda-higher-genus2}.


With the notation in \cite[pp. 360- 361]{shioda-higher-genus2} or similarly
 in \cite[Lemma 5.1]{shioda90}, we can rewrite the above relation as follows:
\[
D - (D\cdot F)O \sim \red(D) - (\red(D)\cdot F)O +
nF + \sum_{v\in \Red(\varphi)}{\mathrm{Corr}_v}(D),
\]
where $n = D\cdot O -\red(D)\cdot O - (D\cdot F  - \red(D)\cdot F)O^2$ and 
\[
\mathrm{Corr}_v(D) = [\Theta_{v,1}, \ldots, \Theta_{v, m_v-1}] \bb_v,
\quad 
\bb_v = \left [\begin{array}{c}
b_{v,1} \\
\vdots \\
b_{v, m_v-1}
\end{array}
\right ]  =  A_v^{-1}
(\vc(v, D) - \vc(v, \red(D)).
\]
Here $A_v$ is the Gram matrix $[\Theta_{v,i}\cdot\Theta_{v, j}]$ and 
\[
\vc(v, \bullet) = 
\left [\begin{array}{c}
\bullet\cdot\Theta_{v, 1} \\
\vdots \\
\bullet\cdot\Theta_{v, m_v-1}
\end{array}
\right ]
\]
for a divisor $\bullet$.

\begin{rem}\label{rem:coefficents}
 Note that $b_{v,i} \in \ZZ$. Since entries of $A_v^{-1}$ are not necessarily 
 integers, this fact gives some constraint on how $D$ and $\red(D)$ intersect with
 $\Theta_{v, i}$'s.
\end{rem}


In order to describe the type of singular fibers, we use the notation in
\cite{namikawa_ueno73}. In this article, there are three possibilities and with suitable
labels for irreducible components $\Theta_{v,i}$ (see the figures below), the corresponding $-A_v$ are as follows:


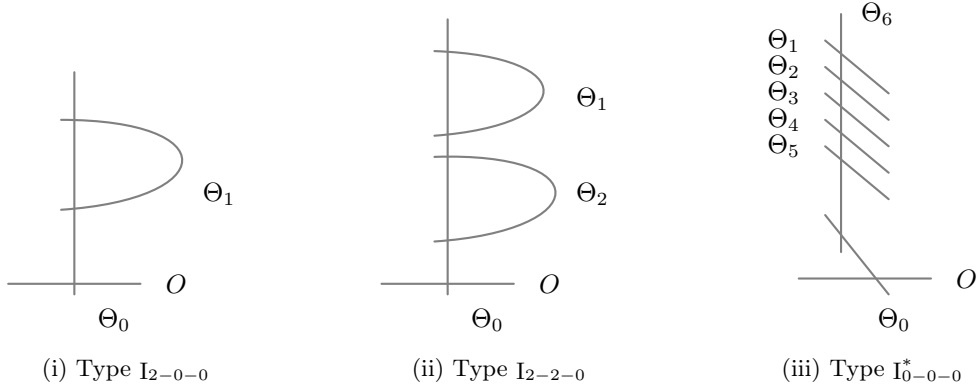
\begin{figure}[h]

\centering

\begin{subfigure}[b]{0.32\textwidth}
\centering
\begin{tikzpicture}[scale=0.7, line cap=round]
\draw[gray, thick] (0.5,0) -- (3,0);
\node[right] at (3.3,0) {$O$};

\draw[gray, thick] (1.75,4.0) -- (1.75,-0.2);
\node[below] at (2.5,-0.3) {$\Theta_0$};
\draw[gray, thick]
(1.5,3.1)
.. controls (4.6,3.1) and (4.5,1.6)
.. (1.5,1.4);
\node[right] at (4.0,1.7) {$\Theta_1$};
\end{tikzpicture}
\caption{Type $\I_{2-0-0}$}
\end{subfigure}
\hfill
\begin{subfigure}[b]{0.32\textwidth}
\centering 
\begin{tikzpicture}[scale=0.7, line cap=round]
\draw[gray, thick] (0.5,0) -- (3,0);
\node[right] at (3.3,0) {$O$};

\draw[gray, thick] (1.75,5.0) -- (1.75,-0.2);
\node[below] at (2.5,-0.3) {$\Theta_0$};
\draw[gray, thick]
(1.5,4.4)
.. controls (4.3,4.3) and (4.2,3.0)
.. (1.5,2.8);
\node[right] at (4.0,3.5) {$\Theta_1$};
\draw[gray, thick]
(1.5,2.4)
.. controls (4.6,2.5) and (4.5,1.0)
.. (1.5,0.8);
\node[right] at (4.0,1.7) {$\Theta_2$};
\end{tikzpicture}
\caption{Type $\I_{2-2-0}$}
\end{subfigure}
\hfill
\begin{subfigure}[b]{0.32\textwidth}
\centering 
\begin{tikzpicture}[scale=0.7, line cap=round]


\draw[gray, thick] (0.5,0) -- (3,0);

\node[right] at (3.3,0) {$O$};


\draw[gray, thick] (1.0,1.2) -- (2.2,-0.3);

\node[below] at (2.3,-0.4) {$\Theta_0$};


\foreach \y/\lab in {4.5/1,4.0/2,3.5/3,3.0/4,2.5/5}
{
\draw[gray, thick] (1.0,\y) -- (2.2,\y-1.0);
\node[left] at (0.7,\y) {$\Theta_{\lab}$};
}

\draw[gray, thick] (1.3,0.5) -- (1.3,5.0);
\node[right] at (1.5,5.0) {$\Theta_6$};
\end{tikzpicture}
\caption{Type $\I^*_{0-0-0}$}

\end{subfigure}
\caption{Possible singular fibers}
\label{figure:singular_fibers}
\end{figure}

\begin{table}[h]
\centering
\begin{tabular}{|c|c|c|c|} \hline
   Type   &   $\I_{2-0-0}$ & $\I_{2-2-0}$ & $\I^*_{0-0-0}$ \\ \hline
   $-A_v^{-1}$ &   $\left [ \frac 12 \right ]$     &   $\displaystyle{\left [\begin{array}{cc}
                     \frac 12 &  0 \\
                      0 & \frac 12
                    \end{array} \right ]}$ & $\displaystyle{
          \left [ \begin{array}{cccccc}
                    1 & \frac 12 & \frac 12 &\frac 12 & \frac 12 & 1 \\
                    \frac 12 & 1 & \frac 12& \frac 12 &  \frac 12 & 1  \\
                    \frac 12 & \frac 12 & 1 & \frac 12 & \frac 12 & 1  \\
                    \frac 12 & \frac 12 & \frac 12 & 1 & \frac 12 & 1 \\
                    \frac 12 & \frac 12 &\frac 12 & \frac 12 & 1 & 1 \\
                    1 & 1 & 1 & 1 & 1 & 2
                    \end{array}
                    \right ]}$  \\ \hline 
\end{tabular}
\caption{List of $-A_v$}
\label{table:gramm_matrix}
\end{table}
Now we summarize the  possibilities of  $\vc(v, \red(D))$ for our later use.
Let $F_v =  \Theta_{v, 0}+ \sum_{i=1}^{m_v - 1}\mu_{v, i}\Theta_{v,i}$ be a singular fiber of type either $\I_{2-0-0}, \I_{2-2-0}$
or  $\I^*_{0-0-0}$.
Since $\red(D)\cdot F_v \le 2$, $\red(D)$ meets either reduced components of $F_v$ at at
most two points  or a component with $\mu_{v, i}= 2$ at one point.
Hence we have the following list:


\begin{table}[h]
\centering
\begin{tabular}{|c|c|} \hline
The type of $F_{v}$ & $\vc(v, \red(D))$ \\ \hline
$\I_{2-0-0}$ &  $\ve_1, \, 2\ve_1$ \\ \hline
$\I_{2-2-0}$ & $\ve_i (i =1, 2), \, \, 2\ve_i (i = 1, 2), \,\,  \ve_1 + \ve_2$ \\ \hline
$\I^*_{0-0-0}$ &  $\ve_6, \,\, \ve_i \, (1 \le i \le 5),\,\,   2\ve_i\, (1 \le i \le 5), 
\ve_i + \ve_j, \,  (1 \le i <  j \le 5)$ \\ \hline
\end{tabular}
\caption{List of $\vc(v, \red(D))$}
\label{table:contribution}
\end{table}

Here, $\ve_i \in \ZZ^{\oplus (m_v-1)}$ denotes the vector whose $i$-th entry is $1$ and all other entries are $0$.
Note that $\mu_{v, i} = 2$ only for $\Theta_6$ in $\I_{0-0-0}^*$ and $\mu_{v,i} = 1$ for the others. 


\subsection{Six lines and a curve of genus \texorpdfstring{$2$}{} over \texorpdfstring{$\CC(t)$}{}}

Let $\mcL$ be the union of six lines in $\PP^2$ with the condition $(\ast)$ in 
the Introduction. Let $S'_{\mcL}$ be 
the double cover branched along $\mcL$. Let $S_{\mcL}$ and  $S_{\mcL, z_o}$ be 
the surfaces as in the Introduction for $\mcB = \mcL$.
In this subsection, we give a more detailed description for $S'_{\mcL}, S_{\mcL}$
and $S_{\mcL, z_o}$. 
Let $[T, X, Z]$ be homogeneous coordinates of $\PP^2$. Let
$L_{\infty} (= L_0), L_i \, (1 \le i \le 5)$ be $6$ lines in $\PP^2$ given 
by defining equations:
\[
 L_{\infty}: l_\infty =Z  =  0, \quad L_i : l_i = X + a_iT + b_iZ = 0 \, (1\le i \le 5).
\]
 Choose a point $z_o \in L_{\infty}\setminus \cup_{i=1}^5L_i$.  A pencil 
$\overline{\Lambda}_{z_o}$ of lines through
$z_o$ gives rise to a pencil $\Lambda_{z_o}$ of curves of genus $2$ on $S_{\mcL}$.
Let $\nu_{z_o}: S_{\mcL, z_o} \to S_{\mcL}$ be the resolution of indeterminacy of $\Lambda_{z_o}$. We denote the 
fibration of genus $2$ curves induced by $\Lambda_{z_o}$ by
$\varphi_{\mcL, z_o}: S_{\mcL, z_o} \to \PP^1$. We denote its generic
fiber by $\mcC_{\mcL,z_o}$, which is a curve of genus $2$ over $\CC(t)$. 




By our construction, there exists a one-to-one correspondence between singular fibers and lines in $\overline{\Lambda}_{z_o}$
that pass through the nodes of $\mcL$. More precisely, the latter consists
of lines $l$ so that
\begin{enumerate}
    \item[(i)] $l$ passes through one node of $\mcL$, 
    \item[(ii)] $l$ passes through two nodes of $\mcL$ and 
    \item[(iii)] $l = L_{\infty}$.
\end{enumerate}

For each case, the type of the corresponding  singular fiber of $\varphi_{\mcL, z_o}$ is as follows:

\begin{center}
    Case (i) $\I_{2-0-0}$, \quad Case (ii) $\I_{2-2-0}$, \quad  Case (iii) $\I^*_{0-0-0}$.
\end{center}

\subsection{Proof of Theorem~\ref{thm:main1}}\label{subsubsec:double_cover_hexad}

Now assume that $L_{\infty}$, $L_i (i = 1, \ldots, 5)$ are bitangents of 
a smooth quartic $\mcQ$ and there exists a cubic $C_3$ such that
$C_3\cap \mcQ$ are the $12$ tangent points between $\mcQ$ and $\mcL$. In this case,
there exists a conic $C_2$ such that $\mcQ + C_2$ is a member of a pencil generated by $2C_3$ and $\mcL$. 
We consider how $C_2$ intersects with $\mcL$.

We first note that none of  $L_{\infty}, L_i \, \,(i = 1, \ldots, 5)$ is  an irreducible component of $C_2$. In fact, if say $L_{\infty}$ is an irreducible
component of $C_2$, then it is also that of $C_3$. This means that
$C_3$ is tangent to $\mcQ$ at $L_{\infty}$. As we assume that $C_3$ meets $\mcQ$
at $12$ distinct points, this is impossible.
Hence we infer that $C_2$ is either
\begin{enumerate}
    \item[(a)] a smooth conic,
    \item[(b)] two distinct lines $L + L'$, $L, L' \neq L_{\infty}, L_i (i = 1, \ldots , 5)$, or
    \item[(c)] a double line $2L$, $L \neq L_{\infty}, L_i (i = 1,\ldots, 5)$.
\end{enumerate}
In order to study how $C_2$ meets $\mcL$, we take homogeneous coordinates of 
$\PP^2$ in the following way:

Choose $z_o \in L_{\infty}\setminus (\mcQ\cup C_3\cup C_2)$. We also choose
homogeneous coordinates so that $z_o = [0, 1, 0]$ and let $t = T/Z, x = X/Z$ be
affine coordinates. Then we may assume that $\mcL, \mcQ$ and  $C_3$  are given by the following affine
equations:
\begin{eqnarray*}
   \mcL:  f_{\mcL}(t,x) & = & \prod_{i=1}^5(x + a_it + b_i) = 0,  \\
    \mcQ: u_{\mcQ}(t,x)  & = & x^4 + \ldots = 0,  \\
    C_3 : v_{C_3}(t, x) & = & c_0x^3+ c_1(t)x^2 + \ldots = 0, c_0 (\neq 0) \in \CC \\
    C_2 : u_{C_2}(t, x) & = & x^2 + \ldots = 0, 
\end{eqnarray*}
where these defining equations satisfy
\[
(\star) \quad v_{C_3}^2 - c_0^2u_{\mcQ}u_{C_2} = f_{\mcL}.
\]
We also define $v_{C_2}$ by $v_{C_2} = v_{C_3} (\bmod\, u_{C_2})$. 
With the these equations, $\mcC_{\mcL, z_o}$ is given by $y^2 = f_{\mcL}$.
Also, from $(\star)$, both $\mcQ$ and $C_2$ are splitting with respect to $\tilde{\pi}_{\mcL}$,  i.e.,the pullbacks $\tilde{\pi}_{\mcL}^*\mcQ$, and $\tilde{\pi}_{\mcL}^*C_2$ are 
of the form 
\[
\tilde{\pi}_{\mcL}^*\mcQ = \mcQ^+ + \mcQ^-, \quad \tilde{\pi}_{\mcL}^*C_2 = C_2^+ + C_2^- + E,
\]
where $\sigma_{\tilde{\pi}_{\mcL}}^*(\mcQ^+)  = \mcQ^-, \sigma_{\tilde{\pi}_{\mcL}}^*(C_2^+)= C_2^-$ and  $E$ is a divisor such that $\Supp(E)$ is contained  in the exceptional set of $\mu_{\mcL}$. Now we go on to consider the above three cases (a), (b) and (c).

\medskip

\underline{Case (a).}
Let $\fd_{\mcQ^+}$ and $\fd_{C_2^+}$ be the divisor on $\mcC_{\mcL, z_o}$ determined
by $Q^+$ and $C_2^+$ respectively. Here we label $Q^+$ and $C_2^+$ 
so that the Mumford representation of $\fd_{\mcQ^+}$ and $\fd_{C_2^+}$ are given by
$(u_{\mcQ}, v_{C_3})$ and $(u_{C_2}, v_{C_2})$, respectively. Then
\[
I(\fd_{\mcQ^+}) = \langle [u_{\mcQ}], [y - v_{C_3}]\rangle, \,\, I(\fd_{C_2^+}) = \langle [u_{C_2}], [y - v_{C_2}]\rangle,  \, \,
I(\fd_{\mcC_2^-}) = \langle[u_{C_2}], [y + v_{C_2}]\rangle.
\]
Also we have $I(\fd_{C_2^+}) = \langle [u_{C_2}], [y - v_{C_2}] \rangle = \langle [u_{C_2}], [y - v_{C_3}]\rangle$.
Since $[v_{C_3} - y][v_{C_3} + y] = [v_{C_3}^2 - f] = [c_0^2u_{\mcQ}u_{C_2}]$, we see 
$\divi([y - v_{C_3}]) = \fd_{Q^+} + \fd_{C_2^+} - 6O$, i.e., $\red(\fd_{\mcQ^+}) = \iota^*\fd_{C_2^+} = \fd_{C_2^-}$. This means
$\red(\mcQ^+) = C_2^-$ and we have 
\[
\mcQ^+ - 4O \sim C_2^- - 2O + 4F + \sum_{v \in \Red(\varphi_{\mcL, z_o})} {\mathrm{Corr}}_v(\mcQ^+), \quad \mathrm{Corr}_v(\mcQ^+) = [\Theta_{v,1}, \ldots, \Theta_{v, m_v-1}] \bb_v
\]
where $\bb_v = A_v^{-1}(\vc(v, \mcQ^+)-\vc(v, C_2^-))$. 

Since $\tilde{\pi}_{\mcL}^*(C_2) = C_2^+ + C_2^- + E$, i.e., $C_2$ is split,
the intersection multiplicity between $C_2$ and $\mcL$ is even for each point
in $C_2\cap \mcL$. Hence we see that $p \in C_2\cap\mcL$ is either a node of
$\mcL$ or $p$ is a smooth point of  $\mcL$ and $C_2$ is tangent to $\mcL$ at $p$.


Now suppose that there exists a node $p \in C_2\cap \mcL$. Let $F_v$ be the singular fiber arising from a line passing through $z_o$ and $p$, and 
let  $\Theta$ be an irreducible component of the singular fiber $F_v$ which arises from $p$. 
Namely, $\Theta$ is either
\[
\Theta_{v,1} \quad \text{if the type of $F_v$ is $\I_{2-0-0}$}, \quad 
\Theta_{v,i}\, (i = 1, 2) \quad \text{if the type of $F_v$ is $\I_{2-2-0}$},
\]
or 
\[
\Theta_{v,i}\, (1 \le i \le 5) \quad \text{if the type of $F_v$ is $\I^*_{0-0-0}$}.
\]
By looking into the process of the canonical resolution $\mu_{\mcL}: S_{\mcL} \to S'_{\mcL}$, we see $C_2^{\pm}\Theta = 1$. 
Summing up, we have Table~\ref{tab:contribution2}, where the same notation as Table~\ref{table:contribution} is used.
    \begin{table}[h]
        \centering
        \begin{tabular}{|c|c|c|} \hline
            The type of $F_v$ & $\vc(v, \mcQ^+)$ & $\vc(v, C_2^-)$  \\ \hline
            $\I_{2-0-0}$ & $\vzero$ & $\ve_1$ \\ \hline 
            $\I_{2-2-0}$ & $\vzero$ & $\ve_i$ $(i = 1, 2)$, $\ve_1 + \ve_2$ \\ \hline
            $\I^*_{0-0-0}$ & $2\ve_6$ & $\ve_i + \ve_j$ $(1 \le i < j \le 5)$\\ \hline
        \end{tabular}
        \caption{List of $\vc(v, \mcQ^+)$ and $\vc(v, C_2^-)$}
        \label{tab:contribution2}
    \end{table}

Note that the value of $\vc(c, C_2^-)$ for type $\I_{0-0-0}^\ast$ must be $\ve_i+\ve_j$ as $\Theta_6$ is included in the branch locus and $C_2$ does not intersect with the branch locus transversely. 
Hence the correction part ${\mathrm{Corr}}_v(\mcQ^+)$ arising from $F_v$ contains 
$\frac 12 \Theta$ or $\frac 32 \Theta$ .  
By Remark~\ref{rem:coefficents}, this is impossible.
Hence $C_2$ is tangent to $\mcL$ at six distinct points.

\medskip

\underline{Case (b).} 
Since both $L$ and $L'$ split, both $L$ and $L'$ do not intersect
any smooth points of $\mcL$. On the other hand, 
both $L$ and $L'$ do not pass any node
of $\mcL$ by the same reason as in Case (a). Hence this case does not occur.

\medskip

\underline{Case (c).} Since $\tilde{\pi}_{\mcL}^*L = L^+ + L^- + E$, $L$ does not intersect $\mcL$ at any smooth point of $\mcL$. Hence $L$ meets $\mcL$ at
three nodes.

The above argument implies  the following proposition:

\begin{prop}\label{prop:salmon}{{\rm \cite[Article 262]{salmon}} Let $\mcQ$ be a smooth plane quartic and let
$L_0, \ldots, L_5$ be a hexad of  bitangents of $\mcQ$ such that
\begin{enumerate}
    \item[\rm (i)] $\sharp(\mcQ\cap(\cup_{i=0}^6L_i))= 12$ and
    \item[\rm (ii)] there exists a cubic $C_3$ through $12$ tangent points.
\end{enumerate}
Then there exists a conic $C_2$ satisfying either (I) $C_2$ is tangent to $\cup_{i=0}^6L_i$ at six distinct points or (II) $C_2$ is a double line $2L$ such
that $L$ passes through three nodes of $\mcL$. 
}
\end{prop}

Now Theorem~\ref{thm:main1} is immediate.


 \section{An explicit construction of a smooth plane quartic with a  preassigned hexad of bitangents}
 \label{sec:construction}


Based on Proposition~\ref{prop:salmon}, there are two cases to be considered as follows:

\begin{itemize}
    \item[(I)] There exists a conic $C_o$ tangent to $\mcL$ at 
    $6$ distinct points.
    \item[(II)] There exists a line $L$  that passes through  
    $3$ distinct nodes of $\mcL$.  
\end{itemize}
For each case, we construct a smooth quartic $\mcQ$ such that $\mcQ + \mcL \in \mcR(\Cmb_{\mcQ\mcL})^{C_3}$.

\underline{Case (I).}  We take new homogeneous coordinates $[T, X, Z]$ of 
$\PP^2$ as follows:

\begin{enumerate}
    \item[(i)] The conic $C_o$ inscribed in  $\mcL$ is given by $XZ - T^2 = 0$ and
    $z_o = [0,1,0]$.
    \item[(ii)] Three of the tangents $L_{\infty}, L_1$ and $L_2$ of $C_o$ are given by
    \[
    L_{\infty}: l_{\infty} = Z = 0, \quad L_4: l_4 = X = 0, \quad L_5: l_5 = X - 2T + Z = 0.
    \]
\end{enumerate}
Note that we can transform any conic and its three tangent lines to the above $C_o$, $L_{\infty}$, $L_1$, $L_2$ in a unique way. Now we choose three points $p_i = (t_i, t^2_i)$  $(i = 1, 2, 3)$ and we denote
the tangent lines at $p_i$ ($i = 1, 2, 3)$ by $L_i: l_i= X - 2t_iT + t_i^2Z = 0$.

Then the generic fiber $\mcC_{\mcL, z_o}$ is given by
\[
\mcC_{\mcL, z_o} : y^2 = l_4l_5\prod_{i=1}^3 l_i = x(x - 2t +1)\prod_{i=1}^3(x - 2t_it + ,t_i^2),
\]
and by our choice of $z_o$, $C_o$ gives rise to a $\CC(t)$-rational point $P_{C_o} = (t^2, t(t-1)\prod_{i=1}^3(t - t_i))$. 

 Under these circumstances, we may assume that the curves $\mcL, C_3$ and
 $C_2$ in Section~\ref{subsubsec:double_cover_hexad} are given as follows:
\begin{eqnarray*}
 \mcL: f_{\mcL}(t, x) & = &l_4l_5\prod_{i=1}^3 l_i =  x(x -2t + 1)\prod_{i=1}^3(x - 2t_it + t_i^2) = 0, \\
 C_3:  v_{C_3}(t, x)  &=&  (a_0 + a_1 t)x^2 + \ldots, = 0, a_0, a_1 \in \CC \\ 
 C_2 = C_o: u_{C_2}(t, x) & = & x - t^2 = 0, 
\end{eqnarray*} 
and we put $v_{C_2^+}(t,x) = t(t-1)\prod_{i=1}^3(t - t_i)$. Note that $C_3$ passes 
through $[0, 1, 0]$.
Since $\red(\mcQ^+) = C_2^-$ and the Mumford representation of $\fd_{C_2^-}$
is $(u_{C_2}, - v_{C_2})$, $v_{C_2} =  v_{C_3} (\bmod \, u_{C_2})$. Hence
we see that there exists a polynomial $q_1(t, x)$ such that
\[
v_{C_3}(t, x) = q_1(t, x)(x - t^2) +  t(t-1)\prod_{i=1}^3(t - t_i).
\]
Here $q(t, x)$ is  of the form $(a_0 + a_1t)x + b(t)$, $b(t) \in \CC[t]$. As $v_{C_3}(t, x)$ is of degree $3$, we have 
\[
a_1 = 1, \quad b(t) = b_0 + b_1t  - (t_1 + t_2 + t_3 + 1)t^2 + t^3, b_i \in \CC (i = 0,1).
\]
By $(\star)$ in Section~\ref{subsubsec:double_cover_hexad}, $c_0 = 1$ and
we obtain $\mcQ$ with $u_{\mcQ} = 0$, where
\[
u_{\mcQ}(t, x) = -\frac {v_{C_3}^2 - f_{\mcL}}{x - t^2}.
\]





\underline{Case (II).}  We take new homogeneous coordinates $[T, X, Z]$ of 
$\PP^2$ as follows:

\begin{enumerate}
    \item[(i)] 
    We choose $L_3, L_4, L_5, L_{\infty}$ so that
    \[
    L_3: l_3 = X + 2T + Z = 0, \,\,L_4: l_4 = X - 2T + Z = 0, \,\, L_5: l_5 = X = 0, \,\,
    L_{\infty}: l_{\infty} = Z = 0.
      \]
      $L_{\infty}$ meets $L_{5}$ at $[1, 0, 0]$ and $L_3$ meets $L_4$ at $[0, -1, 1]$.
    \item[(ii)] We choose the line $L$ in Proposition~\ref{prop:salmon} as
    $X + Z = 0$. This line connects $[1, 0, 0]$ and $[0, -1, 1]$. Choose a point $p = [s, -1, 1]$ $(s \neq 0)$ on $L$ and consider two line $L_1$ and $L_2$ through $p$:
    \[
    L_1: l_1 = X - r_1(T - sZ) + Z = 0, \quad L_2: l_2 = X - r_2(T - sZ) + Z = 0.
    \]
    \item[(iii)] $z_o = [0,1,0]$ and $z_o$ is not any tangent point in $\mcQ\cap \mcL$.
\end{enumerate}
Note that any $6$ lines with no concurrent three lines and with collinear $3$ nodes are obtained as above. 
$\mcC_{\mcL,z_o}$ is given by
\[
\mcC_{\mcL, z_o} : y^2 = l_3l_4l_5\prod_{i=1}^2 l_i = x(x - 2t + 1)(x +2t+1)\prod_{i=1}^2(x - r_i(t - s) + 1),
\]
and $L$ gives a $\CC(t)$-rational point $P_L = (-1, 2\sqrt{r_1}\sqrt{r_2}t(t-s))$.
As $I(P_L) = \langle [x + 1], [y - 2\sqrt{r_1}\sqrt{r_2}t(t-s)]\rangle$, we infer
that
\[
I(2P_L) = \langle [(x + 1)^2], [y - v_{2L}]\rangle,
\]
where 
\[
\frac {1}{\sqrt{r_1}\sqrt{r_2}}t\{(r_1r_2s - r_1r_2t - r_1 - r_2)x + r_1r_2t- r_1r_2s  - r_1 - r_2\}.
\]
Thus we see that $((x+1)^2, v_{2L})$ is the Mumford representation of $2P_L$.

Under these circumstances, we may assume that those curves $C_3$ and
 $C_2$ in Section~\ref{subsubsec:double_cover_hexad} are given as follows:
\begin{eqnarray*}
 \mcL: f_{\mcL}(t, x) & = & l_3l_4l_5l_1l_2 = x(x - 2t + 1)(x +2t+1)\prod_{i=1}^2(x - r_i(t - s) + 1) = 0, \\
 C_3: v_{C_3}(t, x) & = & x^3+ \ldots = 0  \\
 C_2 = C_o: u_{C_2}(t, x) & = & (x+1)^2 = 0, 
\end{eqnarray*} 
and $v_{C_2} = v_{2P_L}$ as $\red(\mcQ^+) = 2L^-$.  Likewise Case (I),
we infer that  there exists a polynomial $q_2(t, x)$ such that
\[
v_{C_3}(t, x) = q_2(t, x)(x+ 1)^2 + v_{2P_L}. 
\]
As $v_{C_3}$ is of degree $3$ in $\CC[t,x]$, we infer that $q_2(t, x) = a_0x + b_0 + b_1t$ for
some $a_0, b_0, b_1 \in \CC$. Again by
$(\star)$ in Section~\ref{subsubsec:double_cover_hexad}, $c_0 = 1$ and we obtain $\mcQ$ as $u_{\mcQ} = 0$, where
\[
u_{\mcQ}(t, x) = \frac{v_{C_3}^2 - f_{\mcL}}{(x+1)^2}.
\]

Note that any smooth quartic with a hexad of bitangents for both Cases (I) and (II) is obtained
as above.

Summing up our observation in this section, we have the following:

\begin{prop}\label{prop:parameter}{For both of Cases (I) and (II),  
$\mcQ + \mcL$ are parameterized with $6$ parameters 
$\{t_1, t_2, t_3, a_0, b_0, b_1\}$ (resp. $\{r_1, r_2, s, a_0, b_0, b_1\}$) for
Case (I) (resp. Case (II)) up to projective 
transformation.
}
\end{prop}

We end this section by giving explicit examples.

\begin{ex}\label{ex:caseI} {An example for Case (I).  Put $t_1 = -2, t_2 = -1, t_3 = 2$, $a_0 = 1,b_0 = -1, b_1 = 2$. Then we have
\begin{eqnarray*}
f_{\mcL}(t,x) & = & l_4l_5l_1l_2l_3 = x(x - 2t +1)(x + 4t + 4)(x + 2t + 1)(x - 4t + 4), \\
v_{C_3}(t, x) &=& (1+t)x^2 - (t-1)^2x - 7t^3 + t^2 + 4t, \\
u_{\mcQ}(t, x) &=& x^4 + (-2t + 9)x^3 + (-13t^2 - 2t + 35)x^2 + (14t^3 - 45t^2 - 4t + 39)x \\
& &  + 49t^4 - 14t^3 - 55t^2 + 8t + 16.  
\end{eqnarray*}
Straightforward computation with a  computer algebraic system, e.g., Maple,  shows that the above equation gives an
example for Case (I).
}    
\end{ex}

\begin{ex}\label{ex:caseII}{An example for Case (II). Put $r_1 = 1, r_2 = 4,s = 1$, $a_0 =1, b_0 = 2, b_1 = 1$. Then we have
\begin{eqnarray*}
    f_{\mcL}(t, x) & =& l_5l_4l_3l_2l_1 =  x(x -2t  + 1)(x + 2t  + 1)(x - t + 2)(x - 4t + 5), \\
    v_{C3}(t, x) & =& x^3 + (t + 4)x^2 + (2t^2 + \frac{5}{2}t + 5)x - 2t^2 + \frac{11}{2}t + 2, \\
    u_{\mcQ}(t, x)&=& x^4 + (2t + 5)x^3 + (5t^2 + 14t + 6)x^2 + (4t^3 + 7t^2 + 34t + 2)x  \\
    &  & + 4t^4 - 22t^3 + \frac{89}{4}t^2 + 22t + 4.
\end{eqnarray*}
  Again straightforward computations with a computer algebra system shows that the above equation gives an
example for Case (II).
}

\end{ex}





\section{Proof of Theorem~\ref{thm:main2}}
\label{sec:proof_theorem}




 In this section, we prove Theorem~\ref{thm:main2}. Let $\mcD_1$ (resp. $\mcD_2$) be the subset of curves of $\mcR(\Cmb_{\mcQ\mcL})^{C_3}$ satisfying (I) (resp. (II)) in Theorem \ref{thm:main2}. From our observation in the previous section, up to projective transformation, both of $\mcD_1$ and
$\mcD_2$ are parametrized by Zariski open sets in $\CC^6$. Hence $\mcD_i$ ($i = 1, 2$) are connected.

We prepare some terminology and notation to discuss the embedded topology of curves in $\mcD_i$.
Given a smooth quartic $\mcQ$, a {\it syzygetic} triad of bitangents is a set of three bitangent lines of $\mcQ$ such that there exists a conic through the six tangent points. On the other hand if such a conic does not exist, the triad is said to be {\it azygetic}.
For $B_2=\mcQ+\mcL\in\mcD_2$, there exsits a line $L$ that passes through three distinct nodes of $\mcL$. We label the six lines in $\mcL$ so that the intersections $L_1\cap L_2$, $L_3\cap L_4$, $L_5\cap L_6$ are the three nodes. Now we have the following:

\begin{lem}\label{lem:triads} 
Under the above setting, the  following holds for curves in $\mcD_i$:
\begin{enumerate}
    \item If $B_1\in \mcD_1$, all triads of bitangents of $B_1$ are azygetic. 
    \item If $B_2\in \mcD_2$, then $B_2$ has exactly eight syzygetic triads. The eight triads are formed by choosing one of the two lines passing through each node. Moreover, there are four ways to partition the six bitangents into pairs of syzygetic triads.
\end{enumerate}
\end{lem}
\begin{proof}
By \cite[Theorem 2.4]{shirane18}, the existence (resp. non-existence) of the conic for a triad is equivalent to the  connected number of the triad with respect to the double cover branched along $\mcQ$ being $2$ (resp. $1$). (See \cite{shirane18} for the definition and details of connected numbers.) The connected number is a topological invariant and invariant under deformation, so the property for a triad to be syzygetic/azygetic is also invariant under deformation. Furthermore, since $\mcD_i$ ($i=1, 2$) are connected, it is enough to prove the statement for explicit examples. 
Let $B_1$ and $B_2$ be the curves given in Examples~\ref{ex:caseI} and \ref{ex:caseII}, 
respectively.
By direct calculations using a computer algebra system, the statements can be confirmed.
Moreover, for $B_2$, once we relabel $L_\infty$ to $L_6$, it can be confirmed that a line $L$ passes through $L_1\cap L_2$, $L_3\cap L_4$, $L_5\cap L_6$ and the eight triads $\{L_i, L_j, L_k\}$   $(i, j, k)=(1, 3, 5)$, $(2, 4, 6)$, $(1, 3, 6)$, $(2, 4, 5)$, $(1, 4, 5)$, $(2, 3, 6)$, $(1, 4, 6)$, $(2, 3, 5)$ are syzygetic.
\end{proof}

\begin{rem}
The bitangent lines of a quartic are related to the {\it theta characteristics} of the quartic. The invariance of the property of a triad being syzygetic/azygetic under deformation can be explained in terms of the invariance of the configuration of theta characteristics of $\mcQ$ also.
\end{rem}

The above Lemma~\ref{lem:triads} is the key to distinguishing the embedded topology of curves in $\mcD_i$.
\begin{prop}~\label{prop:zariski_pair}
Let $B_1\in \mcD_1$ and $B_2\in \mcD_2$. Then $(B_1, B_2)$ is a Zariski pair.
\end{prop}
\begin{proof}
By \cite[Theorem 2.4]{shirane18}, the difference in the existence/non-existence of the conics implies the difference in the connected numbers of the triads with respect to the double cover branched along $\mcQ$. The connected number is a topological invariant, so $(B_1, B_2)$ is a Zariski pair. 
\end{proof}




\begin{rem}
\begin{enumerate}
    \item[(i)] Let $\mcB$ be a curve in $\mcD_1$. Then any four of the  bitangent lines form a azygetic tetrad.  Given a tetrad of lines, by considering the pencil generated by a double conic $2C$ and the tetrad $L_1+L_2+L_3+L_4$, it is easy to construct a quartic that have the four lines as a syzygetic tetrad of bitangents,  but it is relatively difficult for the azygetic case.
Our construction gives such  explicit examples.
    \item[(ii)] For $\mcQ + \mcL \in 
    \mcR(\Cmb_{\mcQ\mcL})$, $\mcQ + \mcL \in \mcR(\Cmb_{\mcQ\mcL})^{C_3}$ if and only if $\mcQ$ is
    splitting with respect to $\pi_{\mcL}$ by Remark~\ref{rem:splitting}.  Hence by \cite{shimada22} we infer that $\mcD_i (i = 1, 2)$ are
    connected components of $\mcR(\Cmb_{\mcQ\mcL})$.
    \item[(iii)]  There are many ways to formulate corollaries from Lemma~\ref{lem:triads}. For example, we can say that for $B=\mcQ+\mcL\in \mcR(\Cmb_{\mcQ\mcL})^{C_3}$, if $B$ has a syzygetic triad
    then the complementary triad is also syzygetic, and there exists a line passing through three of the nodes of $\mcL$.
\end{enumerate}
\end{rem}

\section*{Declaration of AI usage}
The authors used generative AI to draw the configuration of singular fibers in Figure~\ref{figure:singular_fibers}.

\bibliographystyle{spmpsci}
\bibliography{biblio.bib}

@article{absst,
      title={The realization space of a certain conic
            line arrangement of degree $7$ and 
             a $\pi_1$-equivalent {Z}ariski pair}, 
     author={Amram, Meirav and Bannai, Shinzo and
            Shirane, Taketo and Sinichkin, Uriel and Tokunaga, Hiro-o},
      year={2026},
       JOURNAL = {Israel J. of Math.},
  FJOURNAL = {Israel Journal of Mathematics},
       NUMBER = {273},
     PAGES = {901--935},
       DOI = {https://doi.org/10.1007/s11856-025-2860-9},
      
}

@incollection{survey,
    AUTHOR = {Artal Bartolo, Enrique and Cogolludo, Jos\'{e} Ignacio and Tokunaga, Hiro-o},
     TITLE = {A survey on {Z}ariski pairs},
 BOOKTITLE = {Algebraic geometry in {E}ast {A}sia---{H}anoi 2005},
    SERIES = {Adv. Stud. Pure Math.},
    VOLUME = {50},
     PAGES = {1--100},
 PUBLISHER = {Math. Soc. Japan, Tokyo},
      YEAR = {2008},
   MRCLASS = {14H50 (20F36)},
  MRNUMBER = {2409555},
MRREVIEWER = {Meirav Topol-Amram},
       DOI = {10.2969/aspm/05010001},
}

@article {artal94,
    AUTHOR = {Artal Bartolo, Enrique},
     TITLE = {Sur les couples de {Z}ariski},
   JOURNAL = {J. Algebraic Geom.},
  FJOURNAL = {Journal of Algebraic Geometry},
    VOLUME = {3},
      YEAR = {1994},
    NUMBER = {2},
     PAGES = {223--247},
      ISSN = {1056-3911},
   MRCLASS = {14H30 (14E20)},
  MRNUMBER = {1257321},
MRREVIEWER = {Mutsuo Oka},
}

@article {abst2023-1,
    AUTHOR = {Artal Bartolo, E. and Bannai, Sh. and Shirane, T. and Tokunaga, H.},
     TITLE = {Torsion divisors of plane curves and {Z}ariski pairs},
      NOTE = {Translated from Algebra i Analiz {\bf 34} (2022), no. 5},
   JOURNAL = {St. Petersburg Math. J.},
  FJOURNAL = {St. Petersburg Mathematical Journal},
    VOLUME = {34},
      YEAR = {2023},
    NUMBER = {5},
     PAGES = {721--736},
      ISSN = {1061-0022,1547-7371},
   MRCLASS = {14H50},
  MRNUMBER = {4673793},
MRREVIEWER = {Haohao\ Wang},
       DOI = {10.1090/spmj/1776},
}

@article {bannai_ohno,
    AUTHOR = {Bannai, Shinzo and Ohno, Momoko},
     TITLE = {Two-graphs and the embedded topology of smooth quartics and
              its bitangent lines},
   JOURNAL = {Canad. Math. Bull.},
  FJOURNAL = {Canadian Mathematical Bulletin. Bulletin Canadien de
              Math\'ematiques},
    VOLUME = {63},
      YEAR = {2020},
    NUMBER = {4},
     PAGES = {802--812},
      ISSN = {0008-4395,1496-4287},
   MRCLASS = {14H30 (05C65)},
  MRNUMBER = {4176770},
MRREVIEWER = {Juan\ C.\ Migliore},
       DOI = {10.4153/s0008439520000053},
       URL = {https://doi.org/10.4153/s0008439520000053},
}

@article{bst24,
      title={Arithmetic of double covers and the topology of plane curves}, 
     author={Bannai, Shinzo and
            Shirane, Taketo  and Tokunaga, Hiro-o},
      Journal={Contemp. Math.},
      Fjournal = {Contemporary Mathematics},
      Year={2026},
      Volume={843},
      Pages = {83--107},
}

@article {bannai-tokunaga20-2,
    AUTHOR = {Bannai, Shinzo and Tokunaga, Hiro-o},
     TITLE = {Elliptic surfaces of rank one and the topology of cubic-line
              arrangements},
   JOURNAL = {J. Number Theory},
  FJOURNAL = {Journal of Number Theory},
    VOLUME = {221},
      YEAR = {2021},
     PAGES = {174--189},
      ISSN = {0022-314X},
   MRCLASS = {14J27 (14E20 14H52 57M12)},
  MRNUMBER = {4203565},
       DOI = {10.1016/j.jnt.2020.06.005},
}

@article {bty20,
    AUTHOR = {Bannai, Shinzo and Tokunaga, Hiro-o and Yamamoto, Momoko},
     TITLE = {Rational points of elliptic surfaces and the topology of
              cubic-line, cubic-conic-line arrangements},
   JOURNAL = {Hokkaido Math. J.},
  FJOURNAL = {Hokkaido Mathematical Journal},
    VOLUME = {49},
      YEAR = {2020},
    NUMBER = {1},
     PAGES = {87--108},
      ISSN = {0385-4035},
   MRCLASS = {14J27 (14H30 14H50)},
  MRNUMBER = {4105537},
       DOI = {10.14492/hokmj/1591085013},
}

@article{bty25,
      title={The realization spaces of certain conic-line arrangements of degree 7}, 
     author={Bannai, Shinzo and Tokunaga, Hiro-0 and Yorisaki, Emiko},
    journal={J. Math. Soc. Japan},
      year={2025},
      URL={https://www.mathsoc.jp/publication/JMSJ/pdf/JMSJ9475.pdf},
       eprint={}, 
archivePrefix={},
      primaryClass={math.AG}
}

@article {cantor87,
    AUTHOR = {Cantor, David G.},
     TITLE = {Computing in the {J}acobian of a hyperelliptic curve},
   JOURNAL = {Math. Comp.},
  FJOURNAL = {Mathematics of Computation},
    VOLUME = {48},
      YEAR = {1987},
    NUMBER = {177},
     PAGES = {95--101},
      ISSN = {0025-5718},
   MRCLASS = {11Y16 (11G20 14H25 14H40)},
  MRNUMBER = {866101},
MRREVIEWER = {Kevin S. McCurley},
       DOI = {10.2307/2007876},
}

@incollection {catanese-ciliberto,
    AUTHOR = {Catanese, F. and Ciliberto, C.},
     TITLE = {On the irregularity of cyclic coverings of algebraic surfaces},
 BOOKTITLE = {Geometry of complex projective varieties ({C}etraro, 1990)},
    SERIES = {Sem. Conf.},
    VOLUME = {9},
     PAGES = {89--115},
 PUBLISHER = {Mediterranean, Rende},
      YEAR = {1993},
   MRCLASS = {14E20 (14J40)},
  MRNUMBER = {1225590},
MRREVIEWER = {F.\ E. A. Johnson},
}

@book {dolgachev
,
    AUTHOR = {Dolgachev, Igor V.},
     TITLE = {Classical algebraic geometry},
      NOTE = {},
 PUBLISHER = {Cambridge University Press, Cambridge},
      YEAR = {2012},
     PAGES = {xii+639},
      ISBN = {978-1-107-01765-8},
   MRCLASS = {14-02 (14-01)},
  MRNUMBER = {2964027},
MRREVIEWER = {Arnaud\ Beauville},
       DOI = {10.1017/CBO9781139084437},
       URL = {https://doi.org/10.1017/CBO9781139084437},
}

@book {galbraith,
    AUTHOR = {Galbraith, Steven D.},
     TITLE = {Mathematics of public key cryptography},
 PUBLISHER = {Cambridge University Press, Cambridge},
      YEAR = {2012},
     PAGES = {xiv+615},
      ISBN = {978-1-107-01392-6},
   MRCLASS = {94-02 (68P25 94A60)},
  MRNUMBER = {2931758},
MRREVIEWER = {Jos\'{e} Ignacio Farr\'{a}n},
       DOI = {10.1017/CBO9781139012843},
}

@article {horikawa,
    AUTHOR = {Horikawa, Eiji},
     TITLE = {On deformations of quintic surfaces},
   JOURNAL = {Invent. Math.},
  FJOURNAL = {Inventiones Mathematicae},
    VOLUME = {31},
      YEAR = {1975},
    NUMBER = {1},
     PAGES = {43--85},
      ISSN = {0020-9910},
   MRCLASS = {DML},
  MRNUMBER = {1573789},
       DOI = {10.1007/BF01389865},
}

@book {iitaka,
    AUTHOR = {Iitaka, Shigeru},
     TITLE = {Algebraic geometry},
    SERIES = {Graduate Texts in Mathematics},
    VOLUME = {76},
      NOTE = {An introduction to birational geometry of algebraic varieties,
              North-Holland Mathematical Library, 24},
 PUBLISHER = {Springer-Verlag, New York-Berlin},
      YEAR = {1982},
     PAGES = {x+357},
      ISBN = {0-387-90546-4},
   MRCLASS = {14-01 (14-02 14E05)},
  MRNUMBER = {637060},
MRREVIEWER = {Werner Kleinert},
}

@book {mumfordII,
    AUTHOR = {Mumford, David},
     TITLE = {Tata lectures on theta. {II}},
    SERIES = {Modern Birkh\"{a}user Classics},
      NOTE = {Jacobian theta functions and differential equations,
              With the collaboration of C. Musili, M. Nori, E. Previato, M.
              Stillman and H. Umemura,
              Reprint of the 1984 original},
 PUBLISHER = {Birkh\"{a}user Boston, Inc., Boston, MA},
      YEAR = {2007},
     PAGES = {xiv+272},
      ISBN = {978-0-8176-4569-4; 0-8176-4569-1},
   MRCLASS = {14K25 (14H40 32G20)},
  MRNUMBER = {2307768},
       DOI = {10.1007/978-0-8176-4578-6},
}

@incollection {MWZ,
    AUTHOR = {Menezes, Alfred J. and Wu, Yi-Hong and Zuccherato, Robert J.},
     TITLE = {An Elementary Introductio to Hyperelliptic curves},
 BOOKTITLE = {N. Koblitz: Algebraic Aspects of Cryptography},
     PAGES = {157--178},
 PUBLISHER = {Springer-Verlag, Berlin},
      YEAR = {1998},
}

@article {namikawa_ueno73,
    AUTHOR = {Namikawa, Yukihiko and Ueno, Kenji},
     TITLE = {The complete classification of fibres in pencils of curves of
              genus two},
   JOURNAL = {Manuscripta Math.},
  FJOURNAL = {Manuscripta Mathematica},
    VOLUME = {9},
      YEAR = {1973},
     PAGES = {143--186},
      ISSN = {0025-2611,1432-1785},
   MRCLASS = {14D05 (14H20 32J15)},
  MRNUMBER = {369362},
MRREVIEWER = {W.-D.\ Geyer},
       DOI = {10.1007/BF01297652},
       URL = {https://doi.org/10.1007/BF01297652},
}

@book {qing_liu,
    AUTHOR = {Liu, Qing},
     TITLE = {Algebraic geometry and arithmetic curves},
    SERIES = {Oxford Graduate Texts in Mathematics},
    VOLUME = {6},
      NOTE = {Translated from the French by Reinie Ern\'e,
              Oxford Science Publications},
 PUBLISHER = {Oxford University Press, Oxford},
      YEAR = {2002},
     PAGES = {xvi+576},
      ISBN = {0-19-850284-2},
   MRCLASS = {14-01 (11G30 14A05 14A15 14Gxx 14Hxx)},
  MRNUMBER = {1917232},
MRREVIEWER = {C\'icero\ Carvalho},
}

@article {shimada22,
    AUTHOR = {Shimada, Ichiro},
     TITLE = {Zariski multiples associated with quartic curves},
   JOURNAL = {J. Singul.},
  FJOURNAL = {Journal of Singularities},
    VOLUME = {24},
      YEAR = {2022},
     PAGES = {169--189},
      ISSN = {1949-2006},
   MRCLASS = {14H50 (14J26)},
  MRNUMBER = {4495956},
MRREVIEWER = {Krishna\ Hanumanthu},
       DOI = {10.5427/jsing.2022.24g},
       URL = {https://doi.org/10.5427/jsing.2022.24g},
}

@article {shioda90,
    AUTHOR = {Shioda, Tetsuji},
     TITLE = {On the {M}ordell-{W}eil lattices},
   JOURNAL = {Comment. Math. Univ. St. Paul.},
  FJOURNAL = {Commentarii Mathematici Universitatis Sancti Pauli},
    VOLUME = {39},
      YEAR = {1990},
    NUMBER = {2},
     PAGES = {211--240},
      ISSN = {0010-258X},
   MRCLASS = {14J27 (11G05 14C17 14D10 14J05 14J20)},
  MRNUMBER = {1081832},
MRREVIEWER = {Joseph H. Silverman},
}

@article {shioda-higher-genus1,
    AUTHOR = {Shioda, Tetsuji},
     TITLE = {Mordell-{W}eil lattices for higher genus fibration},
   JOURNAL = {Proc. Japan Acad. Ser. A Math. Sci.},
  FJOURNAL = {Japan Academy. Proceedings. Series A. Mathematical Sciences},
    VOLUME = {68},
      YEAR = {1992},
    NUMBER = {9},
     PAGES = {247--250},
      ISSN = {0386-2194},
   MRCLASS = {14H40 (14H05)},
  MRNUMBER = {1202625},
MRREVIEWER = {Min\ Ho\ Lee},
       URL = {http://projecteuclid.org/euclid.pja/1195511629},
}

@incollection {shioda-higher-genus2,
    AUTHOR = {Shioda, Tetsuji},
     TITLE = {Mordell-{W}eil lattices for higher genus fibration over a
              curve},
 BOOKTITLE = {New trends in algebraic geometry ({W}arwick, 1996)},
    SERIES = {London Math. Soc. Lecture Note Ser.},
    VOLUME = {264},
     PAGES = {359--373},
 PUBLISHER = {Cambridge Univ. Press, Cambridge},
      YEAR = {1999},
      ISBN = {0-521-64659-6},
   MRCLASS = {11G10 (11G30 14H40)},
  MRNUMBER = {1714831},
MRREVIEWER = {Takashi\ Ichikawa},
       DOI = {10.1017/CBO9780511721540.015},
       URL = {https://doi.org/10.1017/CBO9780511721540.015},
}

@article {shirane17,
    AUTHOR = {Shirane, Taketo},
     TITLE = {A note on splitting numbers for {G}alois covers and
              {$\pi_1$}-equivalent {Z}ariski {$k$}-plets},
   JOURNAL = {Proc. Amer. Math. Soc.},
  FJOURNAL = {Proceedings of the American Mathematical Society},
    VOLUME = {145},
      YEAR = {2017},
    NUMBER = {3},
     PAGES = {1009--1017},
      ISSN = {0002-9939,1088-6826},
   MRCLASS = {14E20 (14F45 14H30 14H50 32S15 57M12)},
  MRNUMBER = {3589301},
MRREVIEWER = {Francesca\ Vetro},
       DOI = {10.1090/proc/13298},
       URL = {https://doi.org/10.1090/proc/13298},
}

@article {shirane18,
    AUTHOR = {Shirane, Taketo},
     TITLE = {Connected numbers and the embedded topology of plane curves},
   JOURNAL = {Canad. Math. Bull.},
  FJOURNAL = {Canadian Mathematical Bulletin. Bulletin Canadien de
              Math\'{e}matiques},
    VOLUME = {61},
      YEAR = {2018},
    NUMBER = {3},
     PAGES = {650--658},
      ISSN = {0008-4395},
   MRCLASS = {14H30 (14F45 14H50)},
  MRNUMBER = {3831937},
MRREVIEWER = {Fei Ye},
       DOI = {10.4153/CMB-2017-066-5},
}

@article {shirane26,
    AUTHOR = {Shirane, Taketo},
     TITLE = {A note on combinatorial type and splitting invariants of plane curves},
   JOURNAL = {Canad. J. Math.},
  FJOURNAL = {Canadian Journal of Mathematical. Journal Canadien de
              Math\'{e}matiques},
    VOLUME = {On line, TBD},
      YEAR = {2026},
    NUMBER = {},
     PAGES = {},
      ISSN = {},
   MRCLASS = {},
  MRNUMBER = {},
MRREVIEWER = {},
       DOI = {10.4153/S0008414X26102259},
}

@article {takatoku20-2,
    AUTHOR = {Takahashi, Ai and Tokunaga, Hiro-o},
     TITLE = {Representations of divisors on hyperelliptic curves and plane
              curves with quasi-toric relations},
   JOURNAL = {Comment. Math. Univ. St. Pauli},
  FJOURNAL = {Commentarii Mathematici Universitatis Sancti Pauli},
    VOLUME = {70},
      YEAR = {2022},
     PAGES = {11--27},
      ISSN = {0010-258X},
   MRCLASS = {14G50 (13P10 14Q05)},
  MRNUMBER = {4527785},
  URL = {https://doi.org/10.14992/00022409}
}

@article {tokunaga10,
    AUTHOR = {Tokunaga, Hiro-o},
     TITLE = {Geometry of irreducible plane quartics and their quadratic
              residue conics},
   JOURNAL = {J. Singul.},
  FJOURNAL = {Journal of Singularities},
    VOLUME = {2},
      YEAR = {2010},
     PAGES = {170--190},
   MRCLASS = {14H50 (14H30)},
  MRNUMBER = {2763025},
MRREVIEWER = {Meirav Topol-Amram},
       DOI = {10.5427/jsing.2010.2k},
}

@book {salmon,
    AUTHOR = {Salmon, George},
     TITLE = {A treatise on the higher plane curves: intended as a sequel to
              ``{A} treatise on conic sections''},
      NOTE = {3rd ed},
 PUBLISHER = {Chelsea Publishing Co., New York},
      YEAR = {1960},
     PAGES = {xix+395},
   MRCLASS = {50.00 (53.00)},
  MRNUMBER = {115124},
}

\noindent Shinzo BANNAI\\
Department of Applied Mathematics, Faculty of Science, \\
Okayama University of Science, 
1-1 Ridai-cho, Kita-ku, Okayama 700-0005 JAPAN
\\
{\tt bannai@ous.ac.jp}

\medskip

\noindent Hiro-o TOKUNAGA\\
Department of Mathematical  Sciences, Graduate School of Science, \\
Tokyo Metropolitan University, 1-1 Minami-Ohsawa, Hachiohji 192-0397 JAPAN \\
{\tt tokunaga@tmu.ac.jp}
\end{document}